\documentclass[12pt]{article}

\usepackage{fullpage,color,amsthm,amssymb,graphicx,amsmath,xspace,wrapfig,url,cite}
\usepackage{enumerate}
\usepackage{epsfig}
\usepackage{tikz}

 \newtheorem{theorem}{Theorem}[section]
 \newtheorem{lemma}[theorem]{Lemma}
 
 \newtheorem{corollary}[theorem]{Corollary}

 \newtheorem{example}[theorem]{Example}

\newcommand{\fms}{\mathrm{fms}}
\newcommand{\pfms}{\mathrm{pfms}}
\newcommand{\SG}{\mathcal{S}(G)}

\newcommand{\SGp}{\mathcal{S}_+(G)}
\newcommand{\Mp}{\operatorname{M}_+}
\newcommand{\nul}{\operatorname{null}}
\newcommand{\M}{\operatorname{M}}
\DeclareMathOperator{\T}{T}
\DeclareMathOperator{\CC}{cc}

\begin{document}

\title{On the Complexity of the Positive Semidefinite Zero Forcing Number}

\author{
  Shaun Fallat
    \thanks{Department of Mathematics and Statistics, University of Regina,
    Regina, Saskatchewan S4S 0A2, Canada. Research supported in part by an NSERC Discovery Research Grant,
    Application No.: RGPIN-2014-06036.
       Email: \texttt{Shaun.Fallat@uregina.ca}.}$\,$ \thanks{Corresponding Author}
  \and
  Karen Meagher
    \thanks{Department of Mathematics and Statistics, University of Regina,
    Regina, Saskatchewan S4S 0A2, Canada. Research supported in part by an NSERC Discovery Research Grant,
    Application No.: RGPIN-341214-2013
       Email: \texttt{Karen.Meagher@uregina.ca}.}
  \and
  Boting Yang
    \thanks{Department of Computer Science, University of Regina,
        Regina, Saskatchewan S4S 0A2, Canada. Research supported in part by an NSERC Discovery Research Grant, Application No.: RGPIN-2013-261290.
        Email: \texttt{Boting.Yang@uregina.ca}.}
}

\maketitle

\begin{abstract}
  The positive zero forcing number of a graph is a graph parameter
  that arises from a non-traditional type of graph colouring, and is related to a more
  conventional version of zero forcing. 
  We establish a relation between the zero forcing and the fast-mixed searching,
  which implies some NP-completeness results for the zero forcing problem. 
  For chordal graphs much is understood regarding the relationships
  between positive zero forcing and clique coverings. Building upon
  constructions associated with optimal tree covers and forest covers,
  we present a linear time algorithm for computing the positive zero forcing
  number of chordal graphs.
  We also prove that it is NP-complete to determine if a graph has a
  positive zero forcing set with an additional property.

 \vspace{.4cm}


\vspace{.4cm}
\end{abstract}

\section{Introduction}

The zero forcing number of a graph was introduced in \cite{MR2388646}
and related terminology was extended in \cite{zplus}. First and
foremost, the interest in this parameter has been on applying zero
forcing as a bound on the maximum nullities (or, equivalently, the
minimum rank) of certain symmetric matrices associated with graphs,
although this parameter has been considered elsewhere see, for
example, \cite{param, Yeh}.  Independently, physicists have studied
this parameter, referring to it as the graph infection number, in
conjunction with control of quantum systems \cite{burgarth2007full,
  Sev}.

The same notion arises in computer science within the
context of fast-mixed searching~\cite{Yang}.  At its most basic level,
edge search and node search models represent two significant graph
search problems \cite{Megi88,KiPa86}. Bienstock and Seymour~\cite{BS91} 
introduced the mixed search problem that combines the edge search 
and the node search problems. Dyer et al. \cite{DYY08} introduced
the fast search problem. Recently, a fast-mixed search model was introduced in
an attempt to combine fast search and mixed search models~\cite{Yang}. For this
model, we assume that the simple graph $G$ contains a single
\textsl{fugitive}, invisible to the searchers, that can move at any
speed along a ``searcher-free'' path and hides on vertices or along
edges. 
In this case, the minimum number of searchers required to 
capture the fugitive is called the \textsl{fast-mixed search number of
  $G$}. As we will see, the fast-mixed search number and the zero
forcing number of $G$ are indeed equal.

Suppose that $G$ is a simple finite graph with vertex set $V=V(G)$ and edge
set $E=E(G)$.  We begin by specifying a set of initial vertices of the
graph (which we say are coloured black, while all other vertices are
white). Then, using a designated colour change rule applied to these
vertices, we progressively change the colour of white vertices in the
graph to black. Our colouring consists of only two colours (black and
white) and the objective is to colour all vertices black by repeated
application of the colour change rule to our initial set. In general,
we want to determine the smallest set of vertices needed to be black
initially, to eventually change all of the vertices in the graph to
black.

The conventional zero forcing rule results in a partition of the vertices
of the graph into sets, such that each such set induces a path in $G$. 
Further, each of the initial black
vertices is an end point of one of these paths. More recently a
refinement of the colour change rule, called the positive zero forcing
colour change rule, was introduced. Using this rule, the positive
semidefinite zero forcing number was defined (see, for example,
\cite{zplus, ekstrand2011positive, ekstrand2011note}).  When the
positive zero forcing colour change rule is applied to a set of
initial vertices of a graph, the vertices are then partitioned into sets, so that
each such set induces a tree in $G$.

As mentioned above, one of the original motivations for studying these
parameters is that they both provide an upper bound on the
 maximum nullity of both symmetric and positive
semidefinite matrices associated with a graph (see \cite{zplus,
  param}). For a given graph $G=(V,E)$, define
\[ 
\SG = \{ A=[a_{ij}] : A=A^{T}, \; {\rm for} \; i \neq j, a_{ij} \neq 0 \; {\rm if \; and \; only\;  if}\;  \{i,j\} \in E(G)\}
\]
and let $\SGp$ denote the subset of positive semidefinite matrices in $\SG$. 
We use $\nul (B)$ to denote the nullity of the matrix $B$.
The \textsl{maximum nullity} of $G$ is defined to be  
$\M(G) = \max\{\nul (B) : B \in \SG \}$, and , similarly, 
$\Mp(G) = \max \{  \nul (B) : B \in \SGp \},$ is called the \textsl{maximum positive semidefinite nullity of $G$}.

The zero forcing number has been studied under the alias, the
fast-mixed searching number, the complexity of computing the zero
forcing is generally better understood. Consequently, our focus will
be on the algorithmic aspects of computing the positive semidefinite zero
forcing number for graphs. As with most graph parameters, defined in
terms of an optimization problem, these parameters are complicated to
compute in general. However, some very interesting exceptions arise
such as the focus of this paper, chordal graphs, whose positive zero forcing
number can be found in linear time. However, when we consider a variant of the positive
zero forcing problem, called the min-forest problem, we will show that this variant is 
NP-complete even for echinus graphs, which  are a special type of split graphs. Recall that  a
chordal graph is split if its complement is also chordal.

\section{Preliminaries}

Throughout this paper, we only consider finite graphs with
no loops or multiple edges. We use $G=(V,E)$ to denote a graph with
vertex set $V$ and edge set $E$, and we also use $V(G)$ and $E(G)$
to denote the vertex set and edge set of $G$ respectively. We use
$\{u, v\}$ to denote an edge with endpoints $u$ and $v$.
For a graph $G=(V,E)$ and $v \in V$, the
vertex set $\{u: \{u,v\} \in E\}$ is the {\em neighbourhood} of $v$, denoted as $N_G(v)$. For
$V' \subseteq V$, the vertex set $\{x: \{x,y\} \in E, x \in V \setminus V'$ and $y\in V' \}$ is the {\em neighbourhood} of $V'$, denoted as $N_G(V')$.
We use $G[V']$ to denote
the subgraph induced by $V'$, which consists of all
vertices of $V'$ and all of the edges that connect
vertices of $V'$ in $G$.
We use $G-v$ to denote the subgraph
induced by $V \setminus \{v\}$.

Let $G$ be a graph in which every vertex is initially coloured either
black or white. If $u$ is a black vertex of $G$ and $u$ has exactly
one white neighbour, say $v$, then we change the colour of $v$ to
black; this rule is called the \textsl{colour change rule}. In this
case we say ``$u$ forces $v$'' and denote this action by
{$u\rightarrow v$}. Given an initial colouring of $G$, in which
a set of the vertices is black and all other vertices are white, the
\textsl{derived set} is the set of all black vertices, including the
initial set of black vertices, resulting from repeatedly applying the
colour change rule until no more changes are possible. If the derived
set is the entire vertex set of the graph, then the set of initial
black vertices is called a \textsl{zero forcing set}. The \textsl{zero
  forcing number} of a graph $G$ is the size of the smallest zero
forcing set of $G$; it is denoted by $Z(G)$. The procedure of colouring a graph using the
colour rule is called a \textsl{zero forcing process} or simply a
\textsl{forcing process}. A zero forcing process is
called \textsl{optimal} if the initial set of black vertices is a zero
forcing set of the smallest possible size.

If $Z$ is a zero forcing set of a graph $G$, then we may produce a
list of the forces in the order in which they are performed in the
zero forcing process.  This list can then be divided into paths, known
as forcing chains. A \textsl{forcing chain} is a sequence of vertices
$(v_1,v_2,\ldots,v_k)$ such that $v_i\rightarrow v_{i+1}$, for
$i=1,\ldots,k-1$ in the forcing process.  In every step of a forcing
process, each vertex can force at most one other vertex; conversely
every vertex not in the zero forcing set is forced by exactly one
vertex.  Thus the forcing chains that correspond to a zero forcing set
partition the vertices of a graph into disjoint sets, such that each set induces a path. The
number of these paths is equal to the size of the zero forcing set and
the elements of the zero forcing set are the initial vertices of the
forcing chains and hence end points of these paths (see
\cite[Proposition 2.10]{zplus} for more details). We observe that the 
concept of clearing an edge in the fast-mixed model is equivalent to the notion
of a black vertex forcing a unique white neighbour. Hence fast-mixed searching
and the zero forcing colour change rule are equivalent.

The most widely-studied variant of the zero forcing number is called
positive semidefinite zero forcing or the positive zero forcing
number, and was introduced in \cite{zplus}, see also
\cite{ekstrand2011positive} and \cite{ekstrand2011note}. The
positive zero forcing number is also based on a colour change rule
similar to the zero forcing colour change rule.  Let $G$ be a graph
and $B$ a set of vertices; we will initially colour the vertices of
$B$ black and all other vertices white. Let $W_1,\dots,W_k$ be the
sets of vertices in each of the connected components of $G$ after
removing the vertices in $B$.  If $u$ is a vertex in $B$ and $w$ is
the only white neighbour of $u$ in the graph induced by the subset of
vertices $W_i\cup B$, then $u$ can force the colour of $w$ to
black. This is the \textsl{positive colour change rule}.  The
definitions and terminology for the positive zero forcing process,
such as, colouring, derived set, positive zero forcing number etc.,
are similar to those for the zero forcing number, except we use the
positive colour change rule.

The size of the smallest positive zero forcing set of a graph $G$ is
denoted by $Z_+(G)$. Also for all graphs $G$,
since a zero forcing set is also a positive zero forcing set we have
that $Z_+(G) \leq Z(G)$. Moreover, in \cite{zplus} it was shown that
$\Mp(G) \leq Z_{+}(G)$, for any graph $G$.

As noted above, applying the zero forcing colour change rule to the
vertices of a graph produces a path covering of the vertices in that graph. 
Analogously, applying the positive colour change rule produces
a set of vertex disjoint induced trees in the graph,  referred to as
\textsl{forcing trees}.  Next we define a \textsl{positive zero forcing tree cover}.
Suppose $G$ be a graph and let
$Z$ be a positive zero forcing set for $G$. 
Observe that applying the
colour change rule once, two or more vertices can perform forces at
the same time, and a vertex can force multiple vertices from different
components at the same time.  
For each vertex in $Z$, these forces determine a rooted induced tree. The root of 
each tree is the vertex in $Z$ and two vertices are adjacent if one of them forces the other.

In particular, for any 
such tree, starting at the root and following along the edges of will describe the 
chronological list of forces from this root.

More generally, a \textsl{tree covering of a graph} is a family of induced vertex
disjoint trees in the graph that cover all vertices of the graph. The
minimum number of such trees that cover the vertices of a graph $G$ is
the \textsl{tree cover number} of $G$ and is denoted by $\T(G)$. Any set
of zero forcing trees corresponding to an optimal positive zero forcing set
is of size $Z_+(G)$. Hence, for
any graph $G$, we have $\T(G)\leq Z_+(G)$.

Throughout this paper we will use the term \textsl{optimal} to refer to
the object of the smallest possible size.  For example, we will
consider both \textsl{optimal tree covering} and \textsl{optimal zero
  forcing tree covering} (these have the fewest possible number of trees)
and also \textsl{optimal clique covers} (a clique cover with the fewest
cliques).

In this paper we will focus on \textsl{chordal graphs}. A graph is
chordal if it contains no induced cycles on four or more
vertices. Further, we say a vertex $v$ is \textsl{simplicial}, if the
graph induced by the neighbours of $v$ forms a complete graph (or a
clique). If $\{v_1,v_2, \dots, v_n\}$ is an ordering of the vertices
of a graph $G$, such that the vertex $v_i$ is simplicial in the graph
$G - \{v_1,v_2, \dots ,v_{i-1}\}$, then $\{v_1,v_2, \dots,
v_n\}$ is called a \textsl{perfect elimination ordering}. Every chordal graph
has an ordering of the vertices that is a \textsl{perfect elimination
  ordering}, see also \cite{BPe, G}.

In general it is known for any chordal graph $G$, that
$\Mp(G) = |V(G)| - cc(G)$, where $cc(G)$ denotes the fewest number of
cliques needed to cover (or to include) all the edges in $G$ (see
\cite{MN}). 
 From~\cite{zplus} and \cite{psd1} we
  know that for any graph $G$, 
\begin{align} \label{eq:ccbound}
|V(G)| - cc(G) \leq \Mp(G) \leq Z_{+}(G).
\end{align}

This number, $cc(G)$, is often referred to as the \textsl{clique cover number} of the graph $G$. Further inspection of the
work in \cite{MN} actually reveals that, in fact, for any chordal
graph, $cc(G)$ is equal to the \textsl{ordered set number} ($OS(G)$) of
$G$. In \cite{zplus}, it was proved that for any graph $G$, the
ordered set number of $G$ and the positive zero forcing number of $G$
are related and satisfy, $Z_{+}(G)+ OS(G) = |V(G)|$. As a consequence, we
have that $\Mp(G)=Z_{+}(G)$ for any chordal graph $G$, and, in
particular, $Z_{+}(G) = |V(G)|-cc(G)$. So the value of  the positive zero
forcing number of chordal graphs may be deduced by determining the
clique cover number and vice-versa.
 
%
%

\section{The Complexity of Zero Forcing}

Let $G$ be a connected graph. In the
fast-mixed search model, $G$ initially contains no searchers and it
contains only one fugitive who hides on vertices or along edges.
The fugitive is invisible to searchers, and he can move at any rate and 
at any time from one vertex to another vertex along a
searcher-free path between the two vertices.
An edge (resp. a vertex) where the
fugitive may hide is said to be {\em contaminated}, while an edge (resp. a vertex)
where the fugitive cannot hide is said to be {\em cleared}.
A vertex is said to be {\em occupied} if it has a searcher on it.
There are two types
of actions for searchers in each step of the fast-mixed search model:
\begin{enumerate}
\item
a searcher can be placed on a contaminated vertex, or
\item
a search may slide along a contaminated edge $\{u,v\}$ from $u$ to $v$ if $v$ is
contaminated and all edges incident on $u$ except $\{u,v\}$ are cleared.
\end{enumerate}

In the fast-mixed search model, a contaminated edge becomes cleared
if both endpoints are occupied by searchers or if a searcher slides
along it from one endpoint to the other. 
The graph $G$ is {\em cleared} if all edges
are cleared. The minimum number of searchers required to clear $G$
(i.e., to capture the fugitive)
is the {\em fast-mixed search number} of $G$, denoted by $\fms(G)$.
We first show that the fast-mixed search number of a graph is equal to its zero forcing number.

\begin{theorem}\label{thm:fms=Z}
For any graph $G$, $\fms(G)=Z(G)$.
\end{theorem}

\begin{proof}
We first show that $\fms(G) \leq Z(G)$. Let $B$ be a zero forcing set of $G$.
We now construct a fast-mixed search strategy. Initially we place one
searcher on each vertex of $B$. After these placings, all edges whose
two endpoints are occupied are cleared. If $u$ is a black vertex of $G$ 
and $u$ has exactly one white neighbour, say $v$, then 
$u$ must be an occupied vertex, $v$ and $\{u, v\}$ must be contaminated,
and all edges incident with $u$ except $\{u,v\}$ are cleared. 
Thus, we can slide the searcher on vertex $u$
to vertex $v$ along the edge $\{u,v\}$. After this sliding action,  
the edges between $v$ and the occupied neighbours of $v$
are cleared. Similarly, for each forcing action $x\rightarrow y$,
we can construct a fast-mixed search action by sliding the searcher on vertex $x$
to vertex $y$ along the edge $\{x,y\}$. Since $B$ is a zero forcing set of $G$,
the zero forcing process can change all white vertices to black vertices.
Thus, the corresponding fast-mixed search strategy can clear all vertices
and edges of $G$. Hence, $\fms(G) \leq Z(G)$.

We next show that $\fms(G) \geq Z(G)$. 
Let $S$ be a fast-mixed search strategy that clears $G$ 
using $\fms(G)$ searchers. Let $B$ be the set of vertices on which a 
searcher is placed (not slid to). 
Since a searcher can be placed only on a contaminated vertex,
we know that $|B|=\fms(G)$. We colour all vertices of $B$ black and 
all other vertices white. For a sliding action in $S$ that slides a searcher on a vertex, say $x$,
to a vertex, say $y$, along the edge $\{x,y\}$, $x$ is black and $y$ is the only white
neighbour of $x$ at the current stage. Thus, $x$ forces $y$ to black.
So we can construct a zero forcing process that corresponds to the fast-mixed search 
strategy $S$ such that all vertices of $G$ are black when $G$ is cleared by $S$.
Therefore $\fms(G) \geq Z(G)$.
\end{proof}

%



From Theorem~\ref{thm:fms=Z} and \cite[Thms 6.3, 6.5, Cor. 6.6]{Yang}, respectively, we have the following results.

\begin{corollary}\label{cor:NPC1}
Given a graph $G$ and a nonnegative integer $k$, it is NP-complete 
to determine whether $G$ has a zero forcing process with $k$ 
initial black vertices such that all initial black vertices are leaves of $G$.
This problem remains NP-complete for planar graphs with
maximum degree 3.
\end{corollary}


\begin{corollary}\label{cor:NPC2-2}
Given a graph $G$ with $\ell$ leaves, it is NP-complete 
to determine whether $Z(G)= \lceil \ell /2 \rceil$. 
This problem remains NP-complete for
graphs with maximum degree 4.
\end{corollary}


\begin{corollary}\label{cor:NPC3}
Given a graph $G$ and a nonnegative integer $k$, it is NP-complete 
to determine whether $Z(G) \leq k$. 
This problem remains NP-complete
even for 2-connected (biconnected) graphs with maximum degree 4.
\end{corollary}

At the end of this section, we introduce a searching model, which is an extension of the fast-mixed searching, that corresponds to positive zero forcing. This searching model, called the
\emph{parallel fast-mixed searching}, follows the same setting as the fast-mixed searching except that 
the graph may be split into subgraphs after each placing or sliding action, in such a way that these subgraphs may be cleared in a parallel-like fashion.

Initially, $G$ contains no searchers, and so all vertices of $G$ are contaminated. 
To begin, Let ${\cal G}=\{G\}$. After a placing, 
(e.g., place a searcher on a contaminated vertex $u$), the subgraph $G-u$ is the graph induced by the current contaminated vertices. If $G-u$ is not connected, let $G_1, \dots, G_j$ be all of the connected components of $G-u$. We update ${\cal G}$ by replacing $G$ by subgraphs 
$G[V(G_1)\cup \{u\}], \dots, G[V(G_j)\cup \{u\}]$, 
where $u$ is occupied in each subgraph $G[V(G_i)\cup \{u\}]$, $1\leq i\leq j$.
Consider each subgraph $H \in {\cal G}$ that has not been cleared.
After a placing or sliding action, let $X$ be the set of the contaminated vertices in $H$.
If $H[X]$ is not connected, let
$X_1, \dots, X_j$ be the vertex sets of all connected components of $H[X]$. 
We update ${\cal G}$ by replacing $H$ by subgraphs 
$H[X_1\cup N_H(X_1)\}], \dots, H[X_j\cup N_H(X_j)]$, where  
$N_H(X_i)$ is occupied in each subgraph $H[X_i\cup N_H(X_i)\}]$, $1\leq i\leq j$. 
We can continue this searching and branching process until all subgraphs in ${\cal G}$ are cleared.
It is easy to observe that we can arrange the searching process so that
subgraphs in ${\cal G}$ can be cleared in a parallel-like way. 

The graph $G$ is {\em cleared} if all subgraphs of ${\cal G}$
are cleared. The minimum number of placings required to clear $G$ 
is called the {\em parallel fast-mixed search number} of $G$, and is denoted by $\pfms(G)$.

To illustrate the difference between the parallel fast-mixed searching and the fast-mixed searching, 
let $G_k$ be a unicyclic graph, with $k \geq 4$, with vertex set $V=\{v_0, v_1, \dots,  v_{k-1}, v_k\}$ and  edge set $E=\{\{v_0,v_i\}: i=1,\ldots, k \} \cup \{ \{v_{k-1},v_k\}\}$ (see Figure \ref{fig1}). 
\begin{figure}[h]
\begin{center}
\begin{tikzpicture}

\draw [fill] (0,1)  circle (2pt) node [above] {$v_1$};
\draw [fill] (1,1)  circle (2pt) node [above] {$v_2$};
\draw [fill] (2,1)  circle (2pt) node[above] {$v_3$};
\draw [fill] (3,1)   node {$\dots$};

\draw [fill] (4,1)  circle (2pt) node [above] {$v_{k-1}$};
\draw [fill] (5,1)  circle (2pt) node[above] {$v_k$};

\draw [fill] (3,0)  circle (2pt) node [below] {$v_0$};

\draw [thick] (3, 0) -- (0,1);
\draw [thick] (3, 0) -- (1,1);
\draw [thick] (3, 0) -- (2,1);
\draw [thick] (3, 0) -- (4,1);
\draw [thick] (3, 0) -- (5,1);
\draw [thick] (4,1) -- (5,1);

\end{tikzpicture}
\end{center}
\caption{An example of a graph $G$ with $\pfms(G) < \fms(G)$.}\label{fig1}
\end{figure}
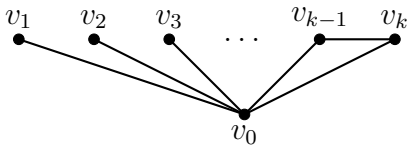
Initially ${\cal G}= \{G_k\}$. After we place a searcher on vertex $v_0$, the set ${\cal G}$ is updated
such that it contains $k-1$ subgraphs, i.e., the edges $\{v_0,v_1\}$, $\dots$,  $\{v_0,v_{k-2}\}$, and the 3-cycle induced by the vertices $\{ v_0, v_{k-1}, v_k\}$, where $v_0$ is occupied by a searcher
in each subgraph. Each edge 
$\{v_0,v_i\}$ ($1 \leq i \leq k-2$), can be cleared by a sliding action.
For the 3-cycle on vertices $\{v_0, v_{k-1}, v_k\}$, since no sliding action can be performed,
we have to place a new searcher on a vertex, say $v_{k-1}$. Since the graph induced by
the contaminated vertices is connected (just  an isolated vertex $v_k$ in this case), 
we do not need to update ${\cal G}$. Now we can slide the searcher on vertex 
$v_{k-1}$ to vertex $v_k$. After the sliding action, the 3-cycle is cleared,
and thus, $G_k$ is cleared. This search strategy contains two placing actions, and it is
easy to see that any strategy with only one placing action cannot cleared $G_k$. Thus 
$\pfms(G_k)=2$. On the other hand, we can easily show that $\fms(G_k)=k-2$.

Similar to Theorem~\ref{thm:fms=Z}, we can prove the following relation between 
the parallel fast-mixed searching and the positive zero forcing.

\begin{theorem}\label{thm:pfms=Z+}
For any graph $G$, $\pfms(G)=Z_+(G)$.
\end{theorem}

\section{A Linear Time Algorithm for Positive Zero Forcing Number of Chordal Graphs}

In this section, we give an algorithm for finding optimal positive
zero forcing sets and optimal tree covers of chordal graphs. Our
algorithm is a modification of the algorithm for computing clique
covers as presented in \cite{ST99}.

\medskip
\noindent
{\bf Algorithm} {\sc Zplus-Chordal} 

\noindent{\em Input}: A connected chordal graph $G$ on $n$ vertices in which all the edges and vertices uncoloured. \\

\noindent{\em Output}: An optimal positive zero forcing tree cover of $G$ and an optimal positive zero forcing set of $G$.

\begin{enumerate}
\item If $n =1$, then mark the single vertex in $G$ black. Output the black vertex set and stop.

\item \label{alg1:s0}
Let $(v_1, v_2, \dots, v_n)$ be the perfect elimination ordering of $G$ given by the lexicographic breadth-first search.
Let $i \leftarrow 1$ and $G_i \leftarrow G$.

\item \label{alg1:s1}
For the simplicial vertex $v_i$ in $G_i$, let $C_i$ be the clique whose vertex
set consists of $v_i$ and all its neighbours in $G_i$.

\item \label{alg1:s2}
If there is an edge $e$ of $G_i$ incident to $v_i$ that is uncoloured, then
\begin{enumerate}
\item colour the edge $e$ black, and colour all other uncoloured edges of the clique $C_i$ red;
\item colour $v_i$ white;
\item go to Step \ref{alg1:s4}.
\end{enumerate}

\item \label{alg1:s3}
If all edges of $G_i$ incident to $v_i$ are coloured, then colour $v_i$ black.

\item \label{alg1:s4}
Set $G_{i+1} \leftarrow G_i - v_i$. 
If $i < |V(G)|-1$, then set $i \leftarrow i+1$ and go to Step \ref{alg1:s1}; 
otherwise, colour the only vertex $v_{i+1}$ in $G_{i+1}$ black and go to Step \ref{alg1:s5}.

\item \label{alg1:s5} Remove all red edges from $G$.  Let
  $\mathcal{T}$ be the set of all connected components of the
  remaining graph after all red edges are removed. Output
  $\mathcal{T}$ and its black vertex set and stop.
\end{enumerate}

Let $V_{\mathrm{black}}(G)$ be the set of all black vertices, then
$V_{\mathrm{white}}(G) = V(G) \backslash V_{\mathrm{black}}(G)$ is the
set of all white vertices from Algorithm {\sc Zplus-Chordal}. The next
result says that the set of black vertices generated with this
algorithm is a positive zero forcing set for the graph. Algorithm {\sc
  Zplus-Chordal} is designed only for connected graphs, but it can be
run on each connected component of a disconnected graph. So, without
loss of generality, we will assume that our graphs are connected.

\begin{lemma}\label{lemma:Z+}
Let $G$ be a connected chordal graph. Then $V_{\mathrm{black}}(G)$ is a positive zero forcing set of $G$.
\end{lemma}

\begin{proof}
  The set $V_{\mathrm{black}}(G)$ is the set of all vertices that are
  initially black in $G$.  Let $V_{\mathrm{white}}(G)=\{w_1, w_2,
  \dots, w_m\}$ where $w_i$ is removed before $w_{i+1}$ in Algorithm
  {\sc Zplus-Chordal} for $1 \leq i < m$. At the iteration when $w_j$
  is coloured white, let $e_j = \{w_j, b_j\}$ be the edge that is
  coloured black. In particular, for the last white vertex $w_m$, the
  black edge is $e_m = \{w_m, b_m\}$ and $b_m$ is in the set
  $V_{\mathrm{black}}(G)$. We claim that in a positive zero forcing
  process in $G$ starting with $V_{\mathrm{black}}(G)$, the vertex
  $b_m$ can force $w_m$.

  Let $H_m$ be the connected component in the subgraph of $G$ induced
  by the vertices $V(G) \backslash V_{\mathrm{black}}(G)$ that
  contains the last white vertex, $w_m$. We will show that the only
  vertex in $H_m$ that is adjacent to $b_m$ is $w_m$. The vertex $b_m$
  is adjacent to $w_m$ and assume that it is also adjacent to another
  vertex, say $u_1$ in $H_m$ (this vertex must be part of
  $V_{\mathrm{white}}(G)$).  Let $\{ b_m, u_1,u_2, \dots, u_k , w_m ,b_m
  \}$ be a cycle of minimal length with $u_1,u_2, \dots, u_k \in
  H_m$. Such a cycle exists since $H_m$ is connected.

  If this cycle has length three, then $u_1$ and $w_m$ are adjacent.
  But at the iteration when $u_1$ is marked white, the edge $\{w_m,
  b_m\}$ will be coloured red.  This is a contradiction, as the
  edge $\{w_m, b\}$ is black.

  Assume this cycle has length more than three and let $u_{\ell}$ be
  the vertex in the cycle that was coloured white first. At the
  iteration where $u_\ell$ is coloured white, it is a simplicial
  vertex. This implies that the neighbours of $u_\ell$ in the cycle
  are adjacent. But this is a contradiction with the choice of $\{
  b_m, u_1,u_2, \dots, u_k , w_m ,b \}$ being a cycle of minimal
  length with $u_1,u_2, \dots, u_k \in H_m$.

  By the positive zero forcing rule, we know that $b_m$ can force
  $w_m$ to be black.  We now change the colour of vertex $w_m$ to
  black, add it to $V_{\mathrm{black}}(G)$ and delete it from
  $V_{\mathrm{white}}(G)$.

  Similarly, at the iteration when $w_{m-1}$ is coloured white, let
  $e_{m-1} = \{w_{m-1},b_{m-1}\}$ be the edge that is coloured
  black. Using the above argument, we can show that $b_{m-1}$ can
  force $w_{m-1}$. Continuing this process, all the white
  vertices are forced to be black. Therefore, $V_{\mathrm{black}}(G)$
  is a positive zero forcing set for $G$.
\end{proof}

For a graph $G$, the set $V_{\mathrm{white}}(G)=\{w_1,w_2, \dots, w_m\}$
is the set of all white vertices produced from applying Algorithm {\sc
  Zplus-Chordal} to $G$ (or on each of its connected components).  Let
$C_{i}$ be the clique whose vertex set consists of $w_i$ and all its
neighbours at the point when $w_i$ is coloured white. Define $C(G)=\{C_1, C_2, \dots, C_m\}$.
 
Every edge of $G$ is coloured in Algorithm {\sc Zplus-Chordal} and at
the iteration when it is coloured, it must belong to some clique
$C_i$. Thus we have the following lemma.

\begin{lemma}\label{lemma:cc}
Let $G$ be a chordal graph. Then $\mathcal{C}(G)$ is a clique cover of $G$.
\end{lemma}

From Lemmas~\ref{lemma:Z+} and \ref{lemma:cc}, we can prove the
correctness of Algorithm {\sc Zplus-Chordal} as follows.

\begin{theorem} \label{thm:correctness}
Let $G$ be a chordal graph. Then $V_{\mathrm{black}}(G)$ is an optimal positive zero forcing set of $G$.
\end{theorem}

\begin{proof}
  Without loss of generality, we suppose that $G$ is a connected
  chordal graph. From Lemma~\ref{lemma:Z+} we know that
  $V_{\mathrm{black}}(G)$ is a positive zero forcing set of
  $G$, so we only need to show that it is the smallest possible.

 From Lemma~\ref{lemma:cc} we know that $C(G)$ is a clique cover of
 $G$. Thus, $cc(G) \leq |C(G)|$, using this with (\ref{eq:ccbound}) we have that
\[
|V(G)| - |C(G)| \leq |V(G)| - cc(G) \leq Z_+(G) \leq | V_{\mathrm{black}}(G)|.
\]

The pair $V_{\mathrm{black}}(G)$ and $V_{\mathrm{white}}(G)$ form a partition of $V(G)$ with
$|V_{\mathrm{white}}(G)|= |C(G)|$, so
\[
|V(G)| - |C(G)| = |V_{\mathrm{black}}(G)|.
\] 
Therefore, $|V(G)| - cc(G) = Z_+(G) = |V_{\mathrm{black}}(G)|$, and
$V_{\mathrm{black}}(G)$ is an optimal positive zero forcing set for
$G$.
\end{proof}

As a byproduct from the proof of Theorem~\ref{thm:correctness}, we
have the following result for chordal graphs.

\begin{corollary} \label{cor:cc}
Let $G$ be a chordal graph. Then
\begin{enumerate}
\item  $\mathcal{C}(G)$ is an optimal clique cover for $G$, \label{cor:ccpart1}
\item $|V(G)| - cc(G) = Z_+(G)$. \label{cor:ccpart2}
\end{enumerate}
\end{corollary}

Note that Corollary~\ref{cor:cc} (\ref{cor:ccpart1}) is proved in
\cite{ST99} by using primal and dual linear
programming. Corollary~\ref{cor:cc} (\ref{cor:ccpart2}) can be deduced
from the work in \cite{psd1}, where the concept of orthogonal removal
is used along with an inductive proof technique.

We can easily modify Algorithm {\sc Zplus-Chordal} so that it can also
output the coloured graph $G$. In this graph, the number of black
edges is equal to the number of white vertices. From the proof of
Lemma~\ref{lemma:Z+}, we know that every black edge can be used to
force a white vertex to black.  Define $T_{\mathrm{black}}(G)$ to be
the subgraph of $G$ formed by taking all the edges (and their
endpoints) that are coloured black by Algorithm {\sc
  Zplus-Chordal}. The next result gives some of the interesting properties of
$T_{\mathrm{black}}(G)$

\begin{theorem}\label{thm:forest}
  Let $G$ be a chordal graph and let $T_{\mathrm{black}}(G)$ be the
  subgraph formed by all the edges  that are
  coloured black in Algorithm {\sc Zplus-Chordal}. Then
\begin{enumerate}
\item the graph $T_{\mathrm{black}}(G)$ is a forest; \label{forest}
\item all of the white vertices of $G$ are contained among the
  vertices of $T_{\mathrm{black}}(G)$; \label{allthewhite}
\item each connected component of $T_{\mathrm{black}}(G)$ contains exactly one black vertex; \label{exactlyonblack}
\item $T_{\mathrm{black}}(G)$ is an induced subgraph of $G$.
\end{enumerate}
\end{theorem}

\begin{proof}

  First, we will simply denote $T_{\mathrm{black}}(G)$ by
  $T_{\mathrm{black}}$.  Note that in each iteration of the Algorithm
  {\sc Zplus-Chordal} when an edge is coloured black in
  Step~\ref{alg1:s2}, it is removed in Step~\ref{alg1:s4} in the same
  iteration. Thus at any iteration $i$, the graph $G_i$ does not
  contain any black edges in Step~\ref{alg1:s1} of the algorithm.

  Suppose there is a cycle in $T_{\mathrm{black}}$.  Let $v_i$ be the
  simplicial vertex in $G_i$ found in Step~\ref{alg1:s1} of Algorithm
  {\sc Zplus-Chordal}, which is the first vertex to be coloured among
  all the vertices on the cycle in $T_{\mathrm{black}}$. Assume that
  $v_i$ is coloured at iteration $i$. Let $u_i$ and $u'_i$ be the two
  neighbours of $v_i$ on the cycle. If $v_i$ is coloured black, then
  both the edges $\{ v_i, u_i\}$ and $\{ v_i, u'_i\} $ must be
  coloured, since none of the edges of $G_i$ can be black, they must
  all be red. But then these edges are not in $T_{\mathrm{black}}$.  If
  $v_i$ is coloured white, then at most one of the edges $\{v_i,
  u_i\}$ and $\{v_i, u'_i\}$ is black, and again, the edges of the
  cycle are not in $T_{\mathrm{black}}$.  Thus $T_{\mathrm{black}}$
  contains no cycles, and hence $T_{\mathrm{black}}$ must be a forest.

  Whenever a vertex of $G$ is coloured white in Step~\ref{alg1:s2} of
  Algorithm {\sc Zplus-Chordal}, an edge incident to it is marked
  black in the same step. Thus $T_{\mathrm{black}}$ contains all the
  vertices that are coloured white by Algorithm {\sc Zplus-Chordal}.

  The vertices of $G$, ordered by the perfect elimination ordering,
  are $\{v_1,v_2, \dots, v_{|V(G)|}\}$. For a connected
  component $T$ in $T_{\mathrm{black}}$, let the vertices in $T$ are
  $V(T)= \{v_{i_1}, v_{i_2}, \dots, v_{i_t}\}$, where $i_1 < i_2 <
  \dots < i_t$, in the perfect elimination ordering of the vertices of
  $G$. We will show that $v_{i_t}$ is the only black vertex in $V(T)$.

  Since $v_{|V(G)|}$ is black, if $i_t = |V(G)|$ then we are done;
  so we will assume that $i_t < |V(G)|$.

  If the vertex $v_{i_t}$ is white, then at Step~\ref{alg1:s2} in the
  algorithm it is coloured white and there is an edge $e=\{v_{i_t},
  v_j\}$ that is coloured black. The edge $e$ must be in $T$ (as it is
  black) and further, $v_j \in V(G_{i_t}) \backslash \{v_{i_t}\}$. Since
  the vertices are being removed in order, this implies that $j>i_t$;
  this is a contradiction since $v_{i_t}$ is the last vertex in the
  ordering that is in $T$. Thus $v_{i_t}$ is coloured black; next we
  will show that $v_{i_{t}}$ is the only vertex in $T$ that is coloured black.

  Assume that $v_{i_j}$ is the vertex with the smallest subscript among
  $\{i_1, i_2, \dots, i_t\}$ that is coloured black in the algorithm.
  Unless $v_{i_j} = v_{i_t}$, the vertex $v_{i_j}$ will be adjacent,
  in $G$, to a vertex in $T$ with a larger subscript.  At
  Step~\ref{alg1:s4} vertex $v_{i_j}$ is removed and all edges
  incident with it are removed as well. So $v_{i_j}$ is not adjacent,
  in $T$, to any vertex with a larger subscript; similarly, no vertex
  with index less than $i_j$ is adjacent in $T$ to a vertex with index
  larger than $i_j$. This implies that $T$ is not connected, which is
  a contradiction.  Therefore, $v_{i_t}$ is the only black vertex in
  $T$.

  Finally, we will show that any connected component $T$ in
  $T_{\mathrm{black}}$ is an induced subgraph of $G$.  If this is not
  the case, then there are two vertices $v_{i_p}$ and $v_{i_q}$ (we
  will assume that $i_p < i_q$), in $V(T)$ such that
  $\{v_{i_p},v_{i_q}\}$ is an edge in $G$ but not in $T$. Pick the
  vertices $v_{i_p}$ and $v_{i_q}$ so that their distance in $T$ is
  minimum over all pairs of vertices that are non-adjacent in $T$, but
  adjacent in $G$.

  At the $i_p$-th iteration of the algorithm, $v_{i_p}$ is the
  simplicial vertex in $G_{i_p}$ which is found in Step
  \ref{alg1:s1}. Let $\{v_{i_p}, v_{i_j}\}$ be the edge marked black
  in Step~\ref{alg1:s2}. Since both $v_{i_j}$ and $v_{i_q}$ are
  adjacent to $v_{i_p}$ in $G_{i_p}$, the edge $\{v_{i_j}, v_{i_q}\}$
  is coloured red at this step. Thus the edge $\{v_{i_j},v_{i_q}\}$ is
  not in $T$. Since $T$ is a tree and $i_p < i_q$, we know that the
  path in $T$ that connects $v_{i_j}$ and $v_{i_q}$ is shorter than
  the path in $T$ that connects $v_{i_p}$ and $v_{i_q}$. Thus we reach
  a contradiction. Hence $T$ is an induced subgraph of $G$.
\end{proof}

From Theorem~\ref{thm:forest}, we have the following result.

\begin{corollary} \label{cor:correctness} Let $G$ be a chordal
  graph. Then the output $\mathcal{T}$ of Algorithm {\sc
    Zplus-Chordal} is an optimal positive zero forcing tree cover of $G$.
\end{corollary}

A chordal graph is called \textsl{non-trivial} if it has at least two
distinct maximal cliques; thus a trivial chordal graph is just a
complete graph. Further, a simplicial vertex is called
\textsl{leaf-simplicial} if none of its neighbours are simplicial. A
tree with only one vertex and no edges is called a \textsl{trivial}
tree. In a positive zero forcing tree cover for a graph,
any tree that is trivial consists of precisely one black vertex.

\begin{lemma}\label{lemma:tc}
Let $G$ be a non-trivial chordal graph.
\begin{enumerate}
\item For any optimal tree cover $\mathcal{T}(G)$, each simplicial
  vertex of $G$ is either a trivial tree in $\mathcal{T}(G)$ or a leaf
  of a non-trivial tree in $\mathcal{T}(G)$.
\item There is an optimal tree cover $\mathcal{T}(G)$ of $G$ such that
  each leaf-simplicial vertex of $G$ is a leaf of a non-trivial tree in
  $\mathcal{T}(G)$.
\end{enumerate}
\end{lemma}

\begin{proof}
  Let $v$ be a simplicial vertex of $G$.  Let $\mathcal{T}(G)$ be an
  optimal tree cover of $G$. Let $T_v$ be the tree in $\mathcal{T}(G)$
  that contains $v$; this means that $T_v$ is an induced tree in
  $G$. The neighbours of $v$ form a clique in $G$, so at most one of
  them can be in $T_v$.  This implies that the degree of $v$ in $T_v$ is
  less than $2$; thus either $T_v$ only contains the vertex $v$ or $v$
  is a leaf in $T_v$.

  Now further assume that $v$ is a leaf-simplicial vertex and suppose
  that $T_v$ is a trivial tree, this is, it only contains the vertex
  $v$. We will show that $\mathcal{T}(G)$ can be transformed into a
  new optimal tree covering of $G$ in which $v$ is a leaf of a
  non-trivial tree. Let $C$ be in the clique in $G$ that contains $v$
  and all its neighbours. We will consider two cases.

  First, assume that no edge of $C$ is also an edge of a tree in
  $\mathcal{T}(G)$.  For any $u \in C$, let $T_u$ be the tree in
  $\mathcal{T}(G)$ that contains $u$. Then $T_v$ and $T_u$ can be
  merged by adding the edge $\{u,v\}$. But this is a contradiction, as
  it implies that $\mathcal{T}(G)$ is not an optimal tree covering of
  $G$.

  Second, assume that there is an edge $\{u, w\}$ in $C$ that is also
  an edge of the tree $T_u$ in $\mathcal{T}(G)$. We can split $T_u$
  into two subtrees by deleting the edge $\{u,w\}$, and then merge
  $T_v$ and one subtree by adding the edge $\{v,u\}$ (the other tree
  will contain $w$, which is not simplicial, so it will not be a tree
  that contains only one leaf-simplicial vertex). Thus we obtain
  another optimal tree cover of $G$, in which $v$ is a leaf in a
  non-trivial tree.  By using the above operation, we can transform all
  trivial trees in $T$ that contain a leaf-simplicial vertex so that
  these vertices are leaves in trees for another optimal tree cover.
\end{proof}

We are now in a position to verify that the positive zero forcing
number can be computed in linear time (in terms of the number of edges
and vertices of $G$), whenever $G$ is chordal.

\begin{theorem} \label{thm:runtime} Let $G$ be a chordal graph with
  $n$ vertices and $m$ edges. Then Algorithm {\sc Zplus-Chordal} can
  be implemented to find an optimal positive zero forcing set of $G$
  and an optimal positive zero forcing tree cover of $G$ in $O(n+m)$ time.
\end{theorem}

\begin{proof}

  The lexicographic breadth-first search algorithm is a linear
  time algorithm that finds a lexicographic ordering of the vertices
  of $G$~\cite{RTL76}. The reverse of a lexicographic ordering of a
  chordal graph is always a perfect elimination ordering. Thus,
  Step~\ref{alg1:s0} requires $O(n+m)$ time.

  In the loop between Step~\ref{alg1:s1} and Step~\ref{alg1:s4} in the
  algorithm, each edge and each vertex is coloured exactly once and
  deleted from the graph once. In Step~\ref{alg1:s2} every edge is
  checked at most once to find uncoloured edges. The total running time of the loop is $O(n+m)$. 

  Finally, it takes linear time to remove all red edges from $G$ in
  Step~\ref{alg1:s5}. Therefore, the running time of the algorithm is
  $O(n+m)$.
\end{proof}

\section{Minimum Forest Covers for Graphs} \label{sec:min-forest}

In this section, we consider the structure of the trees in the zero
forcing tree covers of the graph. So for a given graph $G$ and a
positive integer $\ell$, among all the positive zero forcing tree covers of $G$
with size $\ell$, we want to minimize the number of positive zero forcing trees
that are non-trivial trees. We call this problem \textsl{Min-Forest}. If we
consider the corresponding parallel fast-mixed searching model, each non-trivial
positive zero forcing tree corresponds to an induced tree cleared by a ``mobile''
searcher and each trivial tree corresponds to an ``immobile''
searcher (perhaps a trap or a surveillance camera). Typically, the
goal is to minimize the number of mobile searchers 
among parallel fast-mixed search strategies with a given number of searchers.
The decision version of the Min-Forest problem is as follows.


\begin{quote}
{\sc Min-Forest} \\
{\bf Instance}: A graph $G$ and positive integers $k$ and $\ell$. \\
{\bf Question}: Does $G$ have a positive zero forcing tree cover of size $\ell$
in which there are at most $k$ positive zero forcing trees that are non-trivial?
\end{quote}


A \textsl{split graph} is a graph in which the vertices can be
partitioned into two sets $C$ and $I$, where $C$ induces a clique and
$I$ induces an independent set in the graph. It is not difficult to
show that a graph is split if and only if it is chordal and its
complement is also chordal.

For a graph $G$ a set $U \subset V(G)$ is a \textsl{vertex cover} of
$G$ if every edge of $G$ is incident to at least one vertex in $U$.

\begin{theorem}\label{thm:NPC1} {\sc Min-Forest} is NP-complete. The
  problem remains NP-complete for split graphs whose simplicial
  vertices all have degree $2$.
\end{theorem}

\begin{proof}
It is easy to verify that the problem is in NP. From \cite{GJ79}, we know that the vertex cover problem is NP-complete for cubic graphs. We will construct a reduction from this problem.

  Let $H$ be a cubic graph with vertex set $\{v_1, \dots, v_n\}$ and
  edge set $\{e_1, \dots, e_m\}$.  We construct a connected chordal
  graph $G$ using $H$.  First set
\[
V(G): = \{v_1', \dots, v_n', \; e_1', \dots, e_m', \; x_1, x_2, y\}
\]
where $v_i'$ corresponds to the vertex $v_i$ from $H$ and $e_i'$
corresponds to the edge $e_i$ in $H$. We construct a clique in $G$
with the vertices in the set $\{v_1', \dots, v_n', x_1, x_2\}$. For
each $e_i'$ corresponding to the edge $e_i = \{v_{i_1}, v_{i_2}\}$ in
$H$, we connect the vertex $e_i'$ to vertices $v_{i_1}'$ and
$v_{i_2}'$ in $G$. We finish the construction of $G$ by connecting $y$
to both vertices $x_1$ and $x_2$.  It is easy to see that the graph
$G$ is a connected chordal graph and can be constructed in polynomial
time.

Let $k$ be a positive integer. We will show that $G$ has a zero
forcing tree cover of size $n+1$ in which there are at most $k$
non-trivial trees if and only if there is a vertex cover
of $H$ of size at most $k$.

Initially, suppose that $U \subseteq V(H)$ is a vertex cover of $H$ of size $k$
and let $U'$ be the set of vertices in $G$ that correspond to vertices
of $U$. We will show that the vertices in $U'$ are the only vertices
in a positive zero forcing set for $G$ of size $n+1$ that are the roots for
non-trivial positive zero forcing trees.

Define the set $S=\{e_1', \dots, e_m', y\}$. Every vertex in $S$ is a
simplicial vertex in $G$, so we can assume that these are the first
$m+1$ vertices in the perfect elimination ordering of the
vertices.  Any vertex of $S$ has exactly two neighbours in the clique
induced by $\{v_1', \dots, v_n', x_1, x_2\}$. Since $U$ is a vertex
cover in $H$, every vertex $e_i'$ in $G$ is adjacent to at least one
vertex $v_j'$ in $G$ whose corresponding vertex $v_j$ is in
$U$. Finally, since $y$ is adjacent to $x_1$, each vertex $u \in S$
has at least one neighbour in $U' \cup \{x_1\}$; denote one of these
neighbours by $u'$.

Run Algorithm {\sc Zplus-Chordal} on $G$, assuming that the vertices
in $S$ occur first in the ordering.  For each $u \in S$, at
Step~\ref{alg1:s2} in the algorithm, colour the edge $\{u,u'\}$ black
and colour the vertex $u$ white.  At Step~\ref{alg1:s4} the vertex $u$
is removed and after $m+1$ iterations all vertices of $S$ are
removed. At this point, $x_1$ is a simplicial vertex.  Assume that it
is next in the perfect elimination ordering and at Step~\ref{alg1:s2}
colour the edge $\{x_1, u'\}$ black, where $u'$ is an arbitrary vertex
in $U'$ and in Step~\ref{alg1:s4} the vertex $x_1$ is removed.

At this point all the remaining edges have been coloured and all the
vertices in the set $\{v_1', \dots, v_n', x_2\}$ are coloured black
and form an (optimal) positive zero forcing set for $G$. Let
$V_{\mathrm{black}}(G)$ be the graph formed by taking all edges (and
their endpoints) that were coloured black in Algorithm {\sc
  Zplus-Chordal}. The number of non-trivial components in
$T_{\mathrm{black}}$ is no more than $|U'|=k$. Thus, $G$ has a zero
forcing tree cover of size $n+1$ in which there are at most $k$
non-trivial trees.

Conversely, suppose that $\mathcal{F}$ is a positive zero forcing tree
cover of $G$ with size $n+1$ that contains $k$ non-trivial
trees. Assume that these non-trivial trees are $\{F_1, F_2, \dots,
F_k\}$. From $\mathcal{F}$, we will find a vertex cover of $H$ with size at most $k$.

Let $V'=\{v_1', \dots, v_n', x_1, x_2\}$. 
It is not hard to
see that any set of $n+1$ vertices from $V'$ form a positive zero forcing set for $G$. 

For each non-trivial tree $F_j$ in $\mathcal{F}$, Since it is an induced tree of $G$
and $G[V']$ is a clique, we know that
$F_j$ can contain at most two vertices from $V'$.
If $F_j$ contains two vertices $a$ and $b$ from $V'$,
since each non-trivial tree in $\mathcal{F}$ contains exactly one initial black vertex, 
we know that either $a$ forces $b$ to black or $b$ forces $a$ to black.
Suppose that there are two non-trivial trees in $\mathcal{F}$, say $F_1$ and $F_2$,
such that $F_i$ ($i=1, 2$) contains two vertices $a_i$ and $b_i$ of $V'$ and $a_i$ forces $b_i$ to black.
Suppose that $a_1$ forcing $b_1$ is before $a_2$ forcing $b_2$ in the positive zero forcing process.
Then when $a_1$ forces $b_1$, there are at least two white vertices ($b_1$ and $b_2$) 
in the same component which are adjacent to $a_1$. This contradicts to the colour change rule.
Thus, at most one non-trivial tree in $\mathcal{F}$ that contains two vertices from $V'$.

Since $|V'|=n+2$, we know that among the $n+1$ positive zero forcing trees
in $\mathcal{F}$, only one of them contains two vertices of $V'$ and all others contain only one 
vertex of $V'$. Thus, no vertex in $V(G)\setminus V'$ can form a trivial tree in $\mathcal{F}$.
Without loss of generality, suppose that $F_1$ contains two vertices of $V'$ and 
each $F_i$, $2 \leq i \leq k$, contains exactly one vertex of $V'$. Since each vertex
in $V(G)\setminus V'$ forms a 3-cycle with two vertices in $V'$, we know it must be 
a leaf of a tree in $\{F_1, F_2, \dots, F_k\}$. We have three cases for $F_1$.
\begin{enumerate}
\item
If $F_1$ contains two vertices of 
$\{v_1', \dots, v_n'\}$, then one tree in $\{F_2, \dots, F_k\}$, say $F_k$, consists of only one edge 
that connects $y$ to $x_1$ or $x_2$. So $\{F_1, F_2, \dots, F_{k-1}\}$ contains a subset of $k$ vertices 
$U' \subseteq \{v_1', \dots, v_n'\}$, and each vertex of $\{e_1', \dots, e_m'\}$ is adjacent to one vertex of
$U'$. Thus, the corresponding vertex set of $U'$ in $H$ is a vertex cover of $H$ of size $k$.
\item
If $F_1$ contains one vertex of $\{v_1', \dots, v_n'\}$ and one vertex of $\{x_1, x_2\}$, say $x_1$, 
then one tree in $\{F_2, \dots, F_k\}$, say $F_k$, consists of only one edge 
that connects $y$ to $x_2$.
Thus each tree in $\{F_1, \dots, F_{k-1}\}$ contains one vertex of  $\{v_1', \dots, v_n'\}$.
Let $U' \subseteq \{v_1', \dots, v_n'\}$ be the set of these $k-1$ vertices.  Since each vertex of $\{e_1', \dots, e_m'\}$ is adjacent to one vertex of $U'$, the corresponding vertex set of $U'$ in $H$ is a vertex cover of $H$ of size $k-1$.
\item
If $F_1$ contains both $x_1$ and $x_2$, 
then each tree in $\{F_2, \dots, F_k\}$ contains one vertex of  $\{v_1', \dots, v_n'\}$.
Let $U' \subseteq \{v_1', \dots, v_n'\}$ be the set of these $k-1$ vertices. Since each vertex of $\{e_1', \dots, e_m'\}$ is adjacent to one vertex of $U'$, the corresponding vertex set of $U'$ in $H$ is a vertex cover of $H$ of size $k-1$.
\end{enumerate}
From the above cases, we know that $H$ has
a vertex cover of size at most $k$.
\end{proof}

An \textsl{echinus graph} is a split graph with vertex set $\{C,I\}$,
where $C$ induces a clique and $I$ is an independent set, such that
every vertex of $I$ has two neighbours in $C$ and every vertex of $C$
has three neighbours in $I$. It is easy to see that echinus graphs are
special chordal graphs. From the proof of Theorem~\ref{thm:NPC1}, we
have the following.

\begin{corollary}\label{cor:NPC2}
{\sc Min-Forest} remains NP-complete even for echinus graphs.
\end{corollary}
\begin{proof}
  In the proof of Theorem~\ref{thm:NPC1}, we can modify the
  construction of $G$ by adding two more vertices $y'$ and $y''$ and
  connecting them to vertices $x_1$ and $x_2$, respectively. It is
  easy to see that the new graph $G'$ is an echinus graph. Similarly,
  we can show that $G'$ has a positive zero forcing tree cover of size $n+1$ in
  which there are at most $k$ non-trivial trees, if
  and only if there is a vertex cover of $H$ with size at most $k$.
\end{proof}

If a graph $G$ has a positive zero forcing tree cover of size $\ell$ in which
there are at most $k$ non-trivial trees, then for any $n \geq \ell' > \ell$, $G$
has a positive zero forcing tree cover of size $\ell'$, with at most $k$ non-trivial
trees. Note that the smallest possible value of $\ell$ is the positive zero forcing
number. So next we consider the case when $\ell$ equals the positive zero forcing number.
The following theorem presents a characterization of the chordal graphs for
which there exists an optimal positive zero forcing tree cover that contains
only one non-trivial tree.

\begin{theorem}\label{thm:1-tree}
  Let $G$ be a connected non-trivial chordal graph. There is an optimal
  positive zero forcing tree cover of $G$ in which only one tree
  is non-trivial if and only if for every maximal clique $C$ in $G$,
  there are two vertices $x_C,y_C \in C$ such that any other maximal
  clique $C'$ in $G$ with $V(C') \cap V(C) \neq \emptyset$ must
  contain exactly one of $x_C$ and $y_C$.
\end{theorem}

\begin{proof}
  Let $\mathcal{C}$ be the set of all maximal cliques in $G$. Since $G$
  is connected and non-trivial, each maximal clique in $\mathcal{C}$
  must contain at least two vertices. 

  First assume that for every $C \in \mathcal{C}$, there is a pair of
  vertices $x_C, y_C$ such that each $C' \in \mathcal{C}$ with $V(C')
  \cap V(C) \neq \emptyset$  contains exactly one of
  $x_C$ and $y_C$. We will call the vertices $x_C$ and $y_C$ the
  \textsl{critical vertices for $C$}. 

  If $C$ contains a critical vertex that is only in $C$, then we call
  this vertex a \textsl{representative} of $C$.  For a maximal clique $C
  \in \mathcal{C}$ let $x_C$ and $y_C$ be critical vertices in $C$ and
  assume that $x_C$ is not a representative of $C$. Let $\{C_1,
  \dots, C_i \}$ be the set of all of the maximal cliques, other than
  $C$, in $\mathcal{C}$ that contain $x_C$.  Define $D = C \cap C_1
  \cap \dots \cap C_i$; clearly $x_C \in D$. In fact, any of the
  vertices in $D$, along with $y_C$, forms a pair of critical vertices
  for $C$. So we can fix any vertex $v_D \in D$ to be the
  representative of $D$.  For each $C_j$ with $1\leq j \leq i$, if its
  critical set does not contain $v_D$, then we use $v_D$ to replace
  the critical vertex in $C_j \cap D$.

  In this way we can normalize all critical vertices so that all of
  them are representative critical vertices. In this way we can insure
  that if cliques $C_i$ and $C_j$ have non-trivial intersection, then
  for each of $C_i$ and $C_j$ one of the critical vertices is in
  $C_i \cap C_j$.

  For each clique $C$, colour the edge between the two critical
  vertices black. Let $T_{\mathrm{black}}(G)$ be a graph formed by all
  black edges. Since $G$ is a connected chordal graph and each maximal
  clique of $G$ contains exactly one black edge, we know that $T =
  T_{\mathrm{black}}(G)$ does not contain a cycle.  If $C_i$ and $C_j$
  are adjacent cliques in $G$ then one of the critical vertices for $C_i$
  and one of the critical vertices of $C_j$ are equal. Thus the black
  edges in $C_i$ and in $C_j$ are adjacent. Since $G$ is connected, $T$
  is also connected.

  Next we will show that $T$ along with the empty trees on the
  vertices in $V(G) \setminus V(T)$ is an optimal positive zero forcing tree
  cover of $G$.

  Let $b$ be a leaf of $T$. Set $b$ and all vertices in $V(G)
  \setminus V(T)$ to be the initial set of black vertices in $G$.
  Since $b$ is a leaf of $T$, it is contained in exactly one maximal
  clique in $\mathcal{C}$; call this clique $C_0$. Denote the critical
  vertices in $C_0$ by $\{a, b\}$. All vertices in $V(C_0) \setminus
  \{a\}$ are initially black and $a$ is the unique neighbour of $b$ in
  $T$ after all black vertices are removed. Thus $b$ forces $a$ to
  black.

  Let $C_0, C_1, \dots, C_i \in \mathcal{C}$ be all maximal cliques
  that contain $a$. Assume that the critical vertices for $C_j$ are
  $\{a, a_j\}$, where $j \in \{1, \dots, i\}$. In each $C_j$ with $1
  \leq j \leq i$, all the vertices in $V(C_j) \setminus \{a_j\}$ are
  initially black and so $a_j$ is the unique neighbour of $a$ on the
  component containing $a_j$ after all black vertices are
  removed. Thus $a$ forces $a_j$ to black. It is easy to see that the
  positive zero forcing process can continue until all white vertices
  of $T$ are forced black. Hence, the tree $T$ along with the vertices
  in $V(G) \setminus V(T)$ is a positive zero forcing tree cover of $G$.

  Next we will show that this positive zero forcing tree is optimal.  To do
  this we will show that $cc(G) = |E(T)|$. Since no clique in $G$
  contains two black edges, we know that $cc(G) \geq |E(T)|$. On the
  other hand, each edge of $G$ is contained in a maximal clique in
  $\mathcal{C}$ so we also have that $cc(G) \leq |\mathcal{C}| =
  |E(T)|$.

Now it follows from Corollary~\ref{cor:cc} that 
\[
Z_+(G) = |V(G)| - cc(G) = |V(G)| - |E(T)|  = |(V(G) \setminus V(T))| +1.
\]
Therefore, $\{T\} \cup (V(G) \setminus V(T))$ is an optimal zero
forcing tree cover of $G$.

Conversely we will assume that $Z_+(G) = m$ and that $G$ has an
optimal positive zero forcing tree cover $\mathcal{T} = \{T, v_1, \dots,
v_{m-1}\}$ in which $T$ is the only non-trivial tree. We will show
that every maximal clique has a pair of critical vertices.  We will
show that for every clique $C \in \mathcal{C}$ we have $|V(C) \cap V(T)|
= 2$.  We have three cases to consider.

\begin{enumerate}
\item There is a $C \in \mathcal{C}$ such that $|V(C) \cap V(T)| =
  \emptyset$. 

  The clique $C$ must contain at least two vertices, assume that these
  are $v_1$ and $v_2$. Then we can remove the two isolated vertices
  $v_1$ and $v_2$ from $\mathcal{T}$ and add the edge $\{v_1,v_2\}$ to
  it. This produces a new positive zero forcing tree cover that contains
  $m-1$. This contradicts the optimality of the original tree cover
  $\mathcal{T}$.

\item There is a $C \in \mathcal{C}$ such that $|V(C) \cap V(T)| =
  1$. 

  Let $V(C) \cap V(T) = \{u\}$ and assume that $v_1$ is also in
  $C$. Thus we can remove the isolated vertex $v_1$ from $\mathcal{T}$
  and add the edge $\{u,v_1\}$ to $T$. This produces a positive zero forcing
  tree cover of $G$ of size $m-1$, which is a contradiction.

\item There is a $C \in \mathcal{C}$ such that $|V(C) \cap V(T)| \geq
  3$. 

  Since $C$ is a clique $V(C) \cap V(T)$ is also clique with at least
  $3$ vertices, but this is impossible since $T$ is a tree.

\end{enumerate}

We claim that for every $C \in \mathcal{C}$ the two vertices in
$V(C) \cap V(T)$ form a critical pair of vertices for $C$. 

Let $C$ be an arbitrary maximal clique in $\mathcal{C}$ and let $V(C)
\cap V(T) = \{u, v\}$. Further, let $\{ C_1, \dots, C_i \}$ be the set
of maximal cliques from $\mathcal{C}$ that have nonempty intersection
with $C$. 

Assume that there is a $j \in \{1, \dots, i\}$ such that $V(C_j) \cap
\{u,v\} = \emptyset$. Then the subgraph induced by the vertices in $T$
along with any vertex in $C \cap C_j$ will include an induced cycle
with more than 3 vertices in $G$; but this contradicts the fact that
$G$ is chordal.

Suppose that there is a $j \in \{1, \dots, i\}$ such that $|V(C_j)
\cap \{u,v\} | = 2$. This implies that $u$ and $v$ are both in $C_j$
and that $V(C_j) \cap V(T) = \{u,v\}$. Since $C$ and $C_j$ are both
maximal cliques, there is a vertex $u' \in V(C) \setminus V(C_j)$ and
a vertex $v' \in V(C_j) \setminus V(C)$. Thus both $u'$ and $v'$ are
isolated vertices in $\mathcal{T}$. If we remove the edge $\{u,v\}$
from $T$, then $T$ is split into trees $T_1$ and $T_2$ such that $T_1$
contains $u$ and $T_2$ contains $v$. We then add edge $\{u,u'\}$ to
$T_1$ and add $\{v,v'\}$ to $T_2$ to obtain a new tree cover
$\mathcal{T}$, which is still a positive zero forcing tree cover of $G$
containing $m-1$ trees. This is a contradiction. Hence, $C_j$ contains
exactly one of $u$ and $v$.
\end{proof}

This result can be generalized to a family of graphs that are not
chordal.  

\begin{lemma} \label{inT}
  Let $G$ be a graph and $T$ an induced tree in $G$. If $|V(T)|-1 =
  cc(G)$, then $G$ has an optimal positive zero forcing set with only one
  non-trivial positive zero forcing tree.
\end{lemma}
\proof Colour all the vertices in $V(G) \backslash V(T)$ black and
colour exactly one vertex in $T$ black. This set of black vertices
forms a positive zero forcing set for which there is a positive zero
forcing process where $T$ is the only non-trivial tree.  The size of
this positive zero forcing set is $|V(G)| - |V(T)| + 1 = |V(G)| -
cc(G)$.  Thus, by (\ref{eq:ccbound}), this set is an optimal
positive zero forcing set.  \qed

\begin{example}
{\rm To illustrate Lemma \ref{inT}, consider the graph in Figure \ref{fig2}.

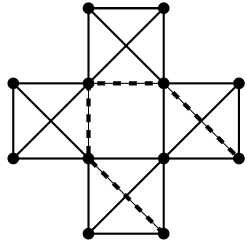
\begin{figure}[h]
\begin{center}
\begin{tikzpicture}

\draw [fill] (1,0)  circle (2pt) node {};
\draw [fill] (2,0)  circle (2pt) node {};

\draw [fill] (0,1)  circle (2pt) node {};
\draw [fill] (1,1)  circle (2pt) node {};
\draw [fill] (2,1)  circle (2pt) node {};
\draw [fill] (3,1)  circle (2pt) node {};

\draw [fill] (0,2)  circle (2pt) node {};
\draw [fill] (1,2)  circle (2pt) node {};
\draw [fill] (2,2)  circle (2pt) node {};
\draw [fill] (3,2)  circle (2pt) node {};

\draw [fill] (1,3)  circle (2pt) node {};
\draw [fill] (2,3)  circle (2pt) node {};

 \draw [thick] (1, 0) -- (2, 0) --(2,1) -- (1,1) -- (1,0) -- (2,1) 
     -- (3,1) -- (3,2) -- (2,2) -- (2,3)--(1,3)--(1,2) -- (0,2)
     -- (0,1) -- (1,1) -- (0,2);

\draw [dashed, ultra thick] (2,0)--(1,1) --(1,2)--(2,2) --(3,1);

\draw (2,0)--(1,1) --(1,2)--(2,2) --(3,1);
 
\draw [thick] (0,1)--(1,2)--(2,3)--(1,3)--(2,2)--(2,1)--(3,2);

\end{tikzpicture}
\end{center}
\caption{An example of a graph $G$ with an induced tree $T$ with
  $|V(T)| - 1 = cc(G)$ } \label{fig2}
\end{figure}

Observe that the clique cover number of $G$ is 4. Consider the induced tree $T$ (based on the
dashed edges in Figure \ref{fig2} containing 5 vertices. Hence $|V(T)|-1 =
  cc(G)$. Following the algorithm in Lemma \ref{inT}, if all remaining vertices plus one 
  are initially coloured black, then $T$ is the only non-trivial tree associated with this
  positive zero forcing tree cover.}
  \end{example}

\section{Cycles of Cliques}

Let $G$ be a graph and assume that $\{C_1, C_2, \dots
,C_k\}$ is a set of maximal cliques in $G$ that covers all the edges
in $G$. We say that $G$ is a \textsl{cycle of cliques} if $V(C_i) \cap
V(C_{j}) \neq \emptyset$ whenever $j = i+1$ or $(i,j) = (k,1)$ and $V(C_i)
\cap V(C_j) = \emptyset$ otherwise.  If $k=1$ then $G$ is a clique; we will
not consider a graph that is a clique to be a cycle of cliques.

\begin{lemma}
  If $G$ is a cycle of cliques $\{C_1, C_2, \dots ,C_k\}$ with $k \geq
  3$, then there is a zero forcing set of size $|V(G)| - (k-2)$ and
  exactly one non-trivial forcing tree.
\end{lemma}
\begin{proof}

To prove this we simply construct a zero forcing set that has this property.
Colour exactly one vertex in $V(C_i) \cap V(C_{i+1})$ white for $i = 1,
\dots , k-2$ and colour all other vertices in $G$ black. This set of
black vertices forms a zero forcing set of size  $|V(G)| - (k-2)$.

Start with any vertex in $V(C_k) \cap V(C_1)$, since all the vertices in
$V(C_k)$ are black, this vertex can force the only one white vertex in $V(C_1)
\cap V(C_2)$. In turn, this new black vertex can force the remaining one white vertex in $V(C_2)
\cap V(C_3)$, which in turn forces the only one white vertex in $C_3 \cap C_4$.
Continuing like this we see that the claim holds.
\end{proof}

If $G$ is a cycle of cliques, then the cliques $\{C_1, C_2, \dots
,C_k\}$ form a clique cover of $G$. So we have that
\[
|V(G)| - |\CC(G)| =  |V(G)| - k \leq Z_+(G) \leq Z(G) \leq |V(G)| - k+2.
\]
Observe that the positive zero forcing sets in the previous lemma may not always be optimal
positive zero forcing sets.  But, in some cases it is possible to
find an optimal zero forcing set for a cycle of cliques that has
exactly one non-trivial forcing tree and is also an optimal positive
zero forcing set.

\begin{lemma}
  Assume that $G$ is a graph that is a cycle of cliques $\{C_1, C_2,
  \dots ,C_k\}$ with $k \geq 3$. Further assume that there 
  is a vertex $x \in V(C_1)$ that is in no other clique 
  and a vertex $y \in V(C_k)$ but is not in any other clique.
  Then there is an optimal
  positive zero forcing set of size $|V(G)|-k$ with exactly one non-trivial
  forcing tree.
\end{lemma}
\begin{proof}

  Colour exactly one vertex in $V(C_i) \cap V(C_{i+1})$ white for each of $i = 1,
  \dots ,k-1$, also colour the vertex $y \in V(C_k)$ white and then
  colour all other vertices in $G$ black. 

The vertex $x \in V(C_1)$ can force the one white vertex in $V(C_1) \cap V(C_2)$. In
turn, this vertex can force the one white vertex in $V(C_2) \cap V(C_3)$,
which in turn forces the one white vertex in $V(C_3) \cap V(C_4)$.  Continue
like this until the one white vertex in $V(C_{k-1})\cap V(C_{k})$ is forced to
black. This vertex can then force $y$ to be black.

Thus this set is a zero forcing set of size $|V(G)| - k$ that has only
one non-trivial zero forcing tree. Since the clique cover number for
this graph is $k$, from (\ref{eq:ccbound}) we have that this is
an optimal zero forcing set and an optimal positive zero forcing set.
\end{proof}

This also gives a family for which the positive zero forcing number
and the forcing number agree.

\section{Further Work}

The complexity of computing any type of graph parameter is an interesting task. For zero forcing
parameters it is known that the problem of finding $Z(G)$ for a graph $G$
is NP-complete. We also suspect that the same is true for computing  $Z_{+}(G)$ 
for a general graph $G$. In fact, we resolve part of this conjecture, by assuming an additional property on
the nature of the zero forcing tree cover that results. However, for chordal graphs $G$, we have 
verified that determining the exact value of $Z_{+}(G)$ can be accomplished via a linear time 
algorithm; the best possible situation. We also believe that it would be
interesting to consider the
complexity of determine $Z_+(G)$ when $G$ is a partial 2-tree. Since for 2-trees it is known that 
the positive zero forcing number is equal to the tree cover number.

A solution of the Min-Forest problem describes an inner structure
between maximal cliques of $G$.  In Section 5, we highlight a couple of instances where
upon if we restrict the number of non-trivial trees in a positive zero forcing 
tree cover of a given size, then conclusions concerning the complexity of computing the positive zero forcing
number can be made. We are interested in exploring this notion
further.


\end{document}